\documentclass[a4paper, 11pt]{amsart}
\usepackage{lmodern}
\usepackage{amsmath, amsthm, amssymb, amsfonts}
\usepackage{nicefrac,mathtools}
\usepackage{graphicx}
\usepackage{tikz-cd}

\usepackage{thmtools}

\numberwithin{equation}{section}
\theoremstyle{plain}
\newtheorem{theorem}[equation]{Theorem}

\newtheorem{lemma}[equation]{Lemma}
\newtheorem{proposition}[equation]{Proposition}
\newtheorem{corollary}[equation]{Corollary}

\theoremstyle{definition}

\theoremstyle{remark} 
\newtheorem{remark}[equation]{Remark}

\newcommand{\N}{\mathbb N}
\newcommand{\A}{\mathcal A}
\newcommand{\B}{\mathcal B}
\newcommand{\D}{\mathcal D}
\newcommand{\E}{\mathcal E}

\newcommand{\K}{\mathcal K}
\newcommand{\G}{\mathcal G}
\newcommand{\red}{\mathrm{red}}
\newcommand{\op}{\mathrm{op}}
\newcommand{\etale}{{\'e}tale}
\renewcommand{\phi}{\varphi}
\newcommand*{\inpro}[2]{\langle#1, #2\rangle}
\newcommand*{\Linpro}[2]{\langle\!\langle#1, #2\rangle\!\rangle}

\title[On Reduced \(C^*\)-algebras of semigroup dynamical systems]{Some remarks on Reduced \(C^*\)-algebras of semigroup dynamical systems and product systems}

\author{Md Amir Hossain}
\address{Indian Statistical Institute, Delhi Centre,  7, S. J. S. Sansanwal Marg, New Delhi, 110016, India}
\email{mdamirhossain18@gmail.com}

\author{S. Sundar}
\address{The Institute of Mathematical Sciences, A CI of Homi Bhabha National Institute, 4th cross street, CIT Campus, Taramani, Chennai, 600113, India.
}
\email{sundarsobers@gmail.com}


\begin{document}
\begin{abstract}
We study the exactness of the reduced crossed product of a semigroup dynamical system and the reduced $C^{*}$-algebra of a product system. We show that for a semigroup dynamical system $(A, P,\alpha)$, under reasonable hypotheses (e.g., $P$ is abelian and finitely generated), the reduced crossed product $A \rtimes_{\mathrm{red}} P$ is exact if and only if $A$ is exact. This strengthens our earlier result (\cite{Amir_Sundar-product-system}), where it was assumed that the action of $P$ on $A$ is by injective endomorphisms. 
 We also compare the groupoid crossed product described in \cite{Amir_Sundar-product-system} and the Fell bundle constructed in \cite{Rennie_Sims} for a product system, and show that they are equivalent as Fell bundles. 
\end{abstract}

\maketitle

\section{Introduction}

This is a continuation of the authors' work (\cite{Amir_Sundar-product-system}), where it was shown that the reduced \(C^*\)-algebra of a proper product system can be viewed as a semigroup crossed product (up to a Morita equivalence) and a few consequences regarding nuclearity, exactness and the invariance of $K$-theory (under homotopy) of the reduced $C^{*}$-algebra of a product system was derived by appealing to the groupoid crossed product picture (\cite{Sundar_Khoshkam}) of the semigroup crossed product. For background and relevant literature,  we refer the reader to \cite{Amir_Sundar-product-system} and the references therein.

The question of exactness for groupoid crossed products and Fell bundle \(C^*\)-algebras has been studied extensively over the past two decades. The exactness for locally compact groupoids was investigated by Anantharaman-Delaroche~\cite{Anantharaman-Delaroche-exact-gpd}, who developed the theory of exact groupoids and established several permanence results.
 Subsequently, Lalonde proved that exactness is preserved under equivalence of groupoids~\cite{Lalonde-2016-Equivalence-exact} and established further permanence properties, including that exactness passes to a broad class of subgroupoids~\cite{Lalonde-Some-permanenece-property-of-exact}.
 The theory was subsequently extended to groupoid crossed products by Lalonde, who studied the exactness of groupoid crossed products associated to exact groupoids. Later, using the stabilization theorem for Fell bundles, he showed in~\cite{Lalonde-stabilization-thm} that reduced \(C^*\)-algebras of Fell bundles over locally compact exact groupoids with Haar systems are exact whenever the \(C_0\)-section algebra of unit fibres is exact. Since the semigroup crossed products considered in this paper admit a realization as a groupoid crossed product,  these results provide a natural framework to analyze the exactness of semigroup crossed products.

In this paper, we prove two results.  First, we improve on our earlier result concerning the exactness of the reduced $C^{*}$-algebra of a proper product system (i.e., the left action of the coefficient algebra on each fibre is by compact operators).   Secondly, we compare the groupoid crossed product picture  (\cite{Amir_Sundar-product-system}) used by the authors with the groupoid Fell bundle constructed by Rennie et al. (\cite{Rennie_Sims}), and we show that they are equivalent as Fell bundles.

In \cite{Amir_Sundar-product-system}, under suitable hypotheses on the Wiener--Hopf groupoid, it was shown that, for a semigroup dynamical system $(A,P,\alpha)$, the reduced crossed product $A \rtimes_{\red} P$ is exact if and only if  $A$ is exact provided that $P$ satisfies the two-sided Ore condition and when the action of $P$ is by injective endomorphisms.
However, there are many examples of semigroup dynamical systems where the action is not by injective endomorphisms (e.g., consider the translation action of $P$ on $c_0(P)$). Also, there are examples of semigroups that are right Ore but not left Ore (e.g., the $ax+b$-semigroup). Thus, it is natural to seek conditions under which  $A \rtimes_{\red} P$ is exact. 

Here, we show that `under suitable amenability and directedness hypotheses', for a semigroup dynamical system $(A,P,\alpha)$, the reduced crossed product $A \rtimes_{\red} P$ is exact if and only if $A$ is exact. Thus, neither the hypothesis that $P$ satisfies the two-sided Ore condition nor the hypothesis that the action of $P$ is injective is required. In particular, this implies that if $P$ is a finitely generated abelian subsemigroup of a group $G$, then $A \rtimes_{\red} P$ is exact if and only if $A$ is exact.
The Morita equivalence mentioned before ensures that, under similar hypotheses, the reduced $C^{*}$-algebra of a proper product system is exact if and only if the coefficient algebra is exact. This improves our earlier result (\cite[Thm. 1.3]{Amir_Sundar-product-system}) where it was assumed that the left action of the coefficient algebra on each fibre is \emph{faithful}.

Our results in \cite{Amir_Sundar-product-system} and also those in this paper rely heavily on the fact that, for a proper product system $X$, up to a Morita equivalence, the reduced $C^{*}$-algebra of $E$, denoted $C_{\red}^{*}(X)$, can be written as a groupoid crossed product, i.e., as a $C^{*}$-algebra associated to a Fell bundle over a groupoid $\mathcal{G}$ called the Wiener--Hopf groupoid. On the other hand, for a quasi-lattice ordered semigroup and for a compactly aligned product system $X$, Rennie et al. (\cite{Rennie_Sims}), also constructed a Fell bundle over the same groupoid $\mathcal{G}$ and showed that the resulting Fell bundle $C^{*}$-algebra is the Nica--Toeplitz algebra of $X$, which coincides with $C_{\red}^{*}(X)$ when the groupoid $\mathcal{G}$ is amenable. It is natural to ask whether the two Fell bundles (when they make sense) are equivalent. We show that this is the case, thereby reconciling the two pictures. We believe that our results are of independent interest, and they are worth recording for future reference.

As far as the organization of this paper is concerned, in addition to this section, the paper has three more sections. In Section~\ref{prelim}, we collect the necessary preliminaries and fix notation. The exactness result is proved in Section~\ref{sec-exactness}, and the comparison of the Fell bundle picture is undertaken in Section~\ref{sec-equiv-Fell-bund}.

\section{Preliminaries}
\label{prelim}
Let $G$ be a countable discrete group, and let $P \subset G$ be a subsemigroup containing the identity element $e$. 
A semigroup dynamical system is a triple \((A,P, \alpha)\) consisting of a \(C^*\)-algebra \(A\) and a  semigroup of endomorphisms \(\alpha =\{\alpha_s\}_{s\in P}\) of \(A\),  i.e., \(s\mapsto \alpha_s\) is a semigroup homomorphism from \(P\to \mathrm{End}(A)\) and \(\alpha_e = \mathrm{id}_A\). Here, $\textup{End}(A)$ denotes the semigroup of endomorphisms of~$A$. We assume that, for each $s \in P$, $\alpha_s$ is non-degenerate, i.e., $\overline{\alpha_s(A)A}=A$. Given a semigroup dynamical system \((A, P, \alpha)\), consider the standard Hilbert \(A\)-module \(\ell^2(P)\otimes A\). Here, \(\otimes\) denotes the external tensor product. For \(x\in A\) and \(s\in P\), define adjointable operators \(\pi(x)\) and \(V_s\) on \(\ell^2(P)\otimes A\) by 
\begin{equation*}
\pi(x)(\delta_t\otimes y) = \delta_t\otimes \alpha_t(x)y \quad \textup{and} \quad V_s(\delta_t\otimes y) =\delta_{ts}\otimes y.
\end{equation*}

The \emph{reduced crossed product} or simply the \emph{semigroup crossed product} is defined as the \(C^*\)-algebra generated by \(\{V_s\pi(x): x\in A, s\in P\}\) and is denoted by \(A\rtimes_{\red}P\). The pair $(\pi,V)$ will be called the left regular representation of $(A,P,\alpha)$, and we sometimes denote $\pi$ by $\pi_A$ to stress the dependence of $\pi$ on $A$. Let $\rho$ be the right regular representation of $G$ and let $w_g$ be the compression of $\rho(g)$ onto $\ell^2(P)$. We say that $(P,G)$ satisfies the \textbf{Toeplitz condition} (see~\cite[Remark 2.5]{Amir_Sundar-product-system}) if $\{w_g \otimes 1:g \in G\}$ is contained in the semigroup  generated by $\{V_s,V_t^*:s,t \in P\} \cup \{0\}$.

We next recall the construction of a groupoid dynamical system from~\cite{Sundar_Khoshkam} for a given semigroup dynamical system \((A,P, \alpha)\). Here, the acting groupoid is called the \emph{Wiener--Hopf groupoid}.
 Denote the power set of \(G\) by \(\mathcal{P}(G)\). We consider $\mathcal{P}(G)$ as a compact metric space by identifying it with $\{0,1\}^{G}$ equipped with the product topology. Consider the right translation action of \(G\) on \(\mathcal{P}(G)\). Define 
 \[
  \Omega := \overline{\{P^{-1}a: a\in P\}} \quad \textup{and} \quad \widetilde{\Omega} :=\bigcup_{g\in G}\Omega g.
 \]
 The Wiener--Hopf groupoid \(\mathcal{G}\) is the reduction of the transformation groupoid \(\widetilde{\Omega}\rtimes G\) to the clopen subset \(\Omega\) of \(\widetilde{\Omega}\), that is,
 \[
  \mathcal{G}: =\widetilde{\Omega}\rtimes G|_{\Omega} = \{(F,g) : F\in \Omega, Fg\in \Omega\}.
 \]
 The groupoid operations are given by \((F,g)(G,h) = (F,gh)\) if and only if \(Fg=G\) and \((F,g)^{-1} = (Fg,g^{-1})\).

  Denote the unitisation of $A$ by $A^{+}$, and let \(\ell^{\infty}(G,A^{+})\) denote the set of all bounded \(A^{+}\)-valued functions on \(G\).  For \(g\in G\) and \(x\in A\), define $j_g(x) \in \ell^{\infty}(G,A^{+})$ by
 \begin{equation*}
 	j_g(x)(h):=\begin{cases}
 		\alpha_{hg^{-1}}(x)& \mbox{ if
 		} hg^{-1} \in P,\cr
 		&\cr
 		0 &  \mbox{ if } hg^{-1} \notin P.
 	\end{cases}
 \end{equation*}
 The right-translation map \(\beta_s(f)(h)=f(hs)\) gives an action of \(G\) on \(\ell^{\infty}(G,A^{+})\) satisfying \(\beta_s(j_g(x)) = j_{gs^{-1}}(x)\). Also, the map \(\phi\colon C_0(\widetilde{\Omega})\to  \ell^{\infty}(G,A^{+})\) defined by 
 \begin{equation*}
 \phi(f)(g) = f(P^{-1}g)
 \end{equation*}
 is a $G$-equivariant embedding. 
 We abuse notation and denote $\phi(f)$ by $f$. 
 Let \[\widetilde{\mathrm{D}}:=C^{*}(\{fj_g(x):f \in C_0(\widetilde{\Omega}), g \in G, x \in A\}).\] Then, $\widetilde{\mathrm{D}}$ is a $G$-$C_0(\widetilde{\Omega})$-algebra. (If $A$ is unital, $\widetilde{\mathrm{D}}$ coincides with the $C^{*}$-algebra generated by $\{j_g(x):g \in G, x \in A\}$.)
 
 We can  write the $G$-$C_0(\widetilde{\Omega})$-algebra $\widetilde{\mathrm{D}}$ as a section algebra of an upper semicontinuous bundle $\widetilde{\mathcal{D}}$ on which the transformation groupoid $\widetilde{\Omega}\rtimes G$  acts. The restriction of the bundle $\widetilde{\mathcal{D}}$ onto the clopen set $\Omega$ is denoted by $\mathcal{D}$, and $\D$ carries an action of  the Wiener--Hopf groupoid $\mathcal{G}=\widetilde{\Omega}\rtimes G|_{\Omega}$. We call the pair $(\mathcal{D},\mathcal{G})$ as the \textbf{groupoid dynamical system associated to $(A,P,\alpha)$}. The main result of \cite{Sundar_Khoshkam}, which we exploited in \cite{Amir_Sundar-product-system}, states that, when $(P,G)$ satisfies the Toeplitz condition,  $A\rtimes_{\red} P$ is isomorphic to the reduced groupoid crossed product $\mathcal{D} \rtimes_{\red} \mathcal{G}$.

The definition of a product system and its associated reduced $C^{*}$-algebra is recalled below. The opposite of the semigroup $P$ is denoted $P^{\op}$. Let $B$ be a $C^{*}$-algebra. A product system of $B$-$B$-correspondences over $P^{\op}$ is a family $X:=\{X_s\}_{s \in P}$ of $B$-$B$-correspondences together with unitaries $U_{s,t}:X_s \otimes X_t \to X_{ts}$, for $s,t \in P$, such that 
\begin{enumerate}
    \item for $s,t \in P$, $U_{s,t}$ is an isomorphism of $B$-$B$-correspondences, and 
    \item for $r,s,t \in P$, $U_{r,ts}(1 \otimes U_{s,t})=U_{sr,t}(U_{r,s} \otimes 1)$.
\end{enumerate}
Throughout this paper, we assume that the fibres are full. A product system $X$ is said to be \emph{proper} if the left action of $B$ on $X_s$, for each $s \in P$, is by compact operators.  

Let $X$ be a product system, and $H:= \bigoplus_{s \in P}X_s$ be the full Fock space. For $u \in X_{s}$, we let $\phi(u)$ be the creation operator defined by \[\phi(u)(\delta_t \otimes v)=\delta_{ts} \otimes U_{s,t}(u \otimes v).\] The $C^{*}$-algebra generated by $\{\phi(u): u \in X_s, s \in P\}$ is called the reduced $C^{*}$-algebra of $X$, and is denoted $C_{\red}^{*}(X)$.

We conclude this section by introducing some additional notation that will be used later.
\begin{itemize}
	\item Throughout this paper,  unless otherwise mentioned the \(C^*\)-algebras that we consider are assumed to be separable.
    \item For $C^{*}$-algebras $A$ and $B$, $A \otimes B$  denotes the spatial tensor product.  
    \item For a semigroup dynamical system $(A,P,\alpha)$ and a $C^{*}$-algebra $I$, the spatial tensor product $I \otimes A$ carries a diagonal action of $P$ where the action of $P$ on $I$ is trivial. This gives us a new semigroup dynamical system which we denote by $(I \otimes A,P, 1\otimes \alpha)$.
    \item For $x,y \in G$, we say $x \leq y$ if $yx^{-1} \in P$. This gives a pre-order on~$G$.  A subset $F \subset G$ is called directed if given $x,y \in F$, there exists $z \in F$ such that $z \geq x, y$.     \item For a product system $X$ over $P^{\op}$ and for $s \in P$, $1_s$ denotes the identity  operator on $X_s$.
    \item For a locally compact Hausdorff space \(Z\), the fibre of a \(C_0(Z)\)-algebra~\(A\) at a point \(z\in Z\) is denoted by \(A(z)\) or  by $A_z$.
\end{itemize}

\section{Exactness of reduced $C^*$-algebras of semigroup dynamical systems}
\label{sec-exactness}

In this section, we prove our first main result: an exactness criterion for the semigroup crossed product and for the reduced \(C^*\)-algebra of a product system. The proof of the main theorem requires some preliminary results, which we establish in a sequence of lemmas and propositions. The first two lemmas show that exactness behaves well with respect to inductive limits and to taking fibres of \(C_0(Z)\)-algebras. These results are likely known, and we have included proofs for completeness.

\begin{lemma}\label{lem-ind-limit-resp-short-ext}
	Let \((A_i, \alpha_{j,i})_{j\geq i}\), \((B_i, \beta_{j,i})_{j\geq i}\) and \((C_i, \gamma_{j,i})_{j\geq i}\) be three inductive systems of \(C^*\)-algebras with inductive limits \((A, \alpha_i)_{i\in I}\), \((B, \beta_i)_{i\in I}\) and \((C, \gamma_i)_{i\in I}\).
	Suppose that for each \(i\), there is a short exact sequence
	\[
	0 \longrightarrow A_i \xlongrightarrow{\phi_i} B_i \xlongrightarrow{\psi_i} C_i\longrightarrow 0
	\]
	of \(C^*\)-algebras, which is compatible with the connecting maps in the sense that  \(\beta_{j,i}\circ \phi_i=\phi_j\circ\alpha_{j,i}\) and \(\gamma_{j,i}\circ\psi_i = \psi_j\circ \beta_{j,i}\) for \(j\geq i\). Then, the induced sequence
	\[
	0 \longrightarrow A \xlongrightarrow{\phi} B \xlongrightarrow{\psi} C\longrightarrow 0
	\]
	is short exact.
\end{lemma}
\begin{proof}
	To prove \(\psi\) is surjective, let \(c\in C\).
 Since \(C=\overline{\bigcup_{i\in I}\gamma_i(C_i)}\) and the image of $\psi$ is closed, it suffices to prove that $c \in \textup{Im}(\psi)$, when $c$ is of the form \(c=\gamma_i(c_i)\) for some $i$. Since  \(\psi_i\) is surjective, there exists \(b_i\in B_i\) such that \(\psi_i(b_i) = c_i\). The definition of \(\psi\) gives us
 \[
 \psi(\beta_i(b_i)) = \gamma_i(\psi_i(b_i))=\gamma_i(c_i)=c.
 \] 
 Hence, $c \in \textup{Im}(\psi)$. This proves that $\psi$ is surjective.

As \(\psi_i\circ \phi_i =0\) for all \(i\in I\), it is clear that $\psi \circ \phi=0$, and hence $\textup{Im}(\phi) \subset \textup{Ker}(\psi)$. To show the other inclusion, let  $b \in B$ be such that  \(\psi(b) =0\). Let $\epsilon>0$ be given. Choose $i$ and $b_i \in B_i$ such that $\|b-\beta_i(b_i)\|<\epsilon$. Since $\psi(b)=0$, 
\[
\|\gamma_i(\psi_i(b_i))\|=\|\psi(b)-\psi(\beta_i(b_i))\|\leq \|b-\beta_i(b_i)\|<\epsilon.
\] 
Hence, for large $j$, $\|\gamma_{j,i}(\psi_i(b_i))\|<\epsilon$. Choose one such $j$ and let $b_j=\beta_{j,i}(b_i)$. Then, 
     \[
     \|\psi_j(b_j)\|=\|\psi_j(\beta_{j,i}(b_i))\|=\|\gamma_{j,i}(\psi_i(b_i))\|<\epsilon.
     \]
Since $\textup{Ker}(\psi_j)=\textup{Im}(\phi_j)$, it follows that there exists $a_j \in A_j$ such that $   \|b_j-\phi_j(a_j)\|<\epsilon$.
Let $a=\alpha_j(a_j)$. Then, \begin{align*}
    \|b-\phi(a)\| &\leq \|b-\beta_j(b_j)\|+\|\beta_j(b_j)-\beta_j(\phi_j(a_j))\|\\& \leq \|b-\beta_i(b_i)\|+\|b_j-\phi_j(a_j)\|\\ & < 2\epsilon.
\end{align*}
As $\textup{Im}(\phi)$ is closed, it follows that $b\in \textup{Im}(\phi)$. Hence, $\textup{Ker}(\psi) \subset \textup{Im}(\phi)$. Consequently, $\textup{Im}(\phi)=\textup{Ker}(\psi)$.

	The argument required to prove that $\phi$ is injective is similar to the one used to prove $\textup{Im}(\phi)=\textup{Ker}(\psi)$. So, we omit the proof. 
\end{proof}

\begin{lemma}\label{lem-fiber-pres-ind-lim}
Let $Z$ be a second countable, locally compact Hausdorff space. 
	Let \(A,B\) and \(C\) be \(C_0(Z)\)-algebras. Let $\phi:A \to B$ and $\psi:B \to C$ be $C_0(Z)$-homomorphisms. If for every \(z\in Z\), the sequence 
	\[
	0 \longrightarrow A(z) \xlongrightarrow{\phi_z} B(z) \xlongrightarrow{\psi_z} C(z)\longrightarrow 0,
	\]
    is exact, then the sequence  
	\[
	0 \longrightarrow A \xlongrightarrow{\phi} B \xlongrightarrow{\psi} C\longrightarrow 0.
	\]
    is exact.
\end{lemma}
\begin{proof}
	Let $a \in A$ be such that \(\phi(a) =0\). Then, \(\phi_z(a(z)) =0\) for all \(z\in Z\). Since each \(\phi_z\) is injective, \(a(z)=0\) for all \(z\in Z\), hence \(a=0\). Thus, \(\phi\) is injective.

	To show \(\psi\) is onto, note that \(\psi(B)\) is a \(C_0(Z)\)-subalgebra of \(C\) as \(\psi\) is a \(C_0(Z)\)-linear map. For \(z\in Z\),
	\[
	\overline{\textup{span}\{\psi(b)(z) : b\in B\}} = \overline{\textup{span}\{\psi_z(b(z)) : b\in B\}} =C(z),
	\]
	because \(\psi_z\) is onto. Then, by~\cite[Lemma A. 4]{Muhly-Williams-Disintegration-theorem}, we conclude that \(\psi\) is onto.
	
	 As \(\psi_z\circ \phi_z =0\) for all \(z\in Z\), we have \(\psi\circ \phi = 0\). Hence, \(\textup{Im}(\phi) \subseteq \textup{Ker}(\psi)\). 
     
Let $b \in B$ be such that $\psi(b)=0$.  
         First, consider the case when $b$ is compactly supported.  Let $K \subset Z$ be compact such that $b$ vanishes outside~$K$. Let $\epsilon>0$. Let $z \in Z$ be given. Since $\psi$ is $C_0(Z)$-linear, $\psi_z(b(z))=0$. Since $\textup{Im}(\phi_z)=\textup{Ker}(\psi_z)$, there exists $\widetilde{a}_z \in A(z)$ such that  $b(z)=\phi_z(\widetilde{a}_z)$. Choose $a \in A$ such that $a(z)=\widetilde{a}_z$. Then, there exists a neighbourhood of~$z$, say $U_z$, such that $\|b(y)-\phi_y(a(y))\|<\epsilon$ for $y \in U_z$. 
A partition of unity argument implies that there exists a compactly supported section $a \in A$ such that $\|b(y)-\phi_y(a(y))\|<\epsilon$ for all $y \in Z$. Then, $\|b-\phi(a)\| <\epsilon$. Hence, $b \in \overline{\textup{Im}{(\phi)}}=\textup{Im}{(\phi)}$.

Now, let $b \in B$ be such that $\psi(b)=0$.  Let $(f_n)_n$ be an approximate identity in $C_0(Z)$ such that $f_n$ is compactly supported for each $n$. Since $C_0(Z)B$ is total in $B$, it follows that $(f_nb)_n \to b$ in norm. Since $\psi$ is a $C_0(Z)$-homomorphism, it follows that $\psi(f_nb)=0$ for all $n$. From the earlier case, we deduce that $f_nb \in \textup{Im}{(\phi)}$. As $\textup{Im}{(\phi)}$ is closed and $(f_nb)_n \to b$ in norm, it follows that $b \in \textup{Im}{(\phi)}$.
\end{proof}

\begin{lemma}\label{lem-comp-tens-prod}
	Let \((A,P, \alpha)\) be a semigroup dynamical system and let \(I\) be a \(C^*\)-algebra. Then, 
	\(I \otimes(A\rtimes_{\red}P) \cong (I\otimes A)\rtimes_{\red}P\).
\end{lemma}
\begin{proof}
Let $(\pi_A,V)$ be the regular representation of $(A,P,\alpha)$ on $\ell^2(P) \otimes A$ and $(\pi_{I \otimes A},\widetilde{V})$ be the regular representation of $(I\otimes A,P,1 \otimes \alpha)$ on $\ell^2(P)\otimes (I \otimes A)$. 
Note that \(I\otimes (A\rtimes_{\red}P) \subseteq \mathcal{L} (I\otimes \mathcal{E})\), where \(\mathcal{E} =\ell^2(P)\otimes A\). Define a canonical unitary \(U\colon I\otimes (\ell^2(P)\otimes A) \to \ell^2(P)\otimes (I\otimes A)\) by 
\[
U(x\otimes (\delta_t\otimes y)) = \delta_t\otimes (x\otimes y).
\]
 Now, for \(x\in I, a\in A\), we have 
\begin{align*}
&U(x\otimes \pi_A(a))U^*(\delta_t\otimes (z\otimes y)) = U(x\otimes \pi_A(a)) (z\otimes(\delta_t\otimes y))\\
&=U(xz\otimes(\delta_t \otimes\alpha_t(a)y)) = \delta_t \otimes (xz\otimes \alpha_t(a)y)\\
&=\pi_{I\otimes A}(x\otimes a) (\delta_t\otimes (z\otimes y)).
\end{align*}
Thus, \(U(x\otimes \pi_A(a)) U^* = \pi_{I\otimes A}(x\otimes a)\) for \(x\in I\) and \(a\in A\). For \(s\in P\), we have 
\begin{align*}
	U(1\otimes V_s)U^*(\delta_t\otimes (z\otimes y)) &=U(1\otimes V_s) (z\otimes (\delta_t\otimes y))
	=U(z\otimes (\delta_{ts}\otimes y))\\ &=\delta_{ts}\otimes (z\otimes y)
	= \widetilde{V}_s(\delta_t\otimes (z\otimes y)).
\end{align*}
Hence, for $s \in P$, $U(1 \otimes V_s)U^*=\widetilde{V}_s$. 
Thus, $\mathrm{Ad}(U)$ maps the generators of $I \otimes (A \rtimes_{\red} P)$ onto the generators of $(I \otimes A) \rtimes_{\red} P$. Hence,   $\mathrm{Ad}(U)$ is an isomorphism from $I \otimes (A \rtimes_{\red} P)$ onto $(I \otimes A)\rtimes_{\red}P$. 
\end{proof}

\begin{proposition}\label{prop-exact-func}
	Let \((I,P, \alpha), (A,P,\beta)\) and \((B,P, \gamma)\) be semigroup dynamical systems. Assume that $A$ and $B$ are unital. Suppose that  \((P, G)\) satisfies the Toeplitz condition, every element of \(\widetilde{\Omega}\) is directed, and that the Wiener--Hopf groupoid \(\mathcal{G}\) is amenable. Then, a \(P\)-equivariant short exact sequence 
	\begin{equation}\label{eq-prop-exact-seq-1}
	0 \longrightarrow I \longrightarrow A \longrightarrow B \longrightarrow 0
		\end{equation} 
	gives the following short exact sequence
	\[
		0 \longrightarrow I\rtimes_{\red} P \longrightarrow A\rtimes_{\red} P \longrightarrow B \rtimes_{\red} P\longrightarrow 0.
	\]
\end{proposition}
\begin{proof}
Let $(\mathcal{D}^{I},\G)$, $(\mathcal{D}^{A},\G)$ and $(\mathcal{D}^{B},\G)$ be the groupoid dynamical system associated to $(I,P,\alpha)$, $(A,P,\beta)$ and $(B,P,\gamma)$, respectively. We also consider the equivalent ones \((\widetilde{\mathcal{D}}^I, \widetilde{\Omega}\rtimes G), (\widetilde{\mathcal{D}}^A, \widetilde{\Omega} \rtimes G)\) and \((\widetilde{\mathcal{D}}^B, \widetilde{\Omega}\rtimes G)\), where the acting groupoid is the transformation groupoid $\widetilde{\Omega}\rtimes G$. 
The section algebra of $\widetilde{\mathcal{D}}^{\bullet}$ is denoted $\widetilde{\mathrm{D}}^{\bullet}$,  and that of $\mathcal{D}^{\bullet}$ is denoted $\mathrm{D}^{\bullet}$ for \(\bullet \in \{I, A, B\}\).

	Let \(F\in \widetilde{\Omega}\). By assumption \(F\) is directed.  Consider the inductive systems of $C^{*}$-algebras
	\[
	  I_s =I, \quad A_s =A, \quad B_s=B
	\]
	with connecting maps
	\[
	 \phi_{t,s} = \alpha_{st^{-1}}, \quad \beta_{st^{-1}}, \quad \gamma_{st^{-1}} \quad \textup{for } t \leq s\in F,
	\]
	respectively. The natural map from $I_s \to  \varinjlim_{s \in F}I_s$ will be denoted by $\lambda_s$. We use the same letter $\lambda_s$ to denote the natural maps $A_s \to  \varinjlim_{s \in F}A_s$ and $B_s \to  \varinjlim_{s \in F}B_s$.  
    
    We claim that the fibre \(\widetilde{\mathcal{D}}^I_{F}\) of the bundle \(\widetilde{\mathcal{D}}^I\) is isomorphic to the inductive limit \(\varinjlim_{s\in F} I_s \). The analogous statement (i.e., the unital version) for $\widetilde{\mathcal{D}}^{A}_F$ and $\widetilde{\mathcal{D}}^{B}_F$ is ~\cite[Prop. 6.2]{Amir_Sundar-product-system}. 
    
    Fix $F \in \widetilde{\Omega}$. We embed $\ell^{\infty}(G,I^{+})$ inside $\ell^{\infty}(G,A^{+})$ and abuse notation. 
    As $A$ is unital, 
	 \begin{align*}
	\widetilde{\mathrm{D}}^A &= C^*(\{j_g(x): x\in A, g\in G\}), \quad \textup{and}\\
\widetilde{\mathrm{D}}^I &= C^*(\{\phi j_g(x): x\in I, g\in G, \phi \in C_0(\widetilde{\Omega})\}).
\end{align*} 
 Since \(\widetilde{\mathrm{D}}^I\) is an ideal of \(\widetilde{\mathrm{D}}^A\), we have the following commutative diagram
  \[
  \begin{tikzcd}
  	\widetilde{\mathrm{D}}^I \arrow[r, hook] \arrow[d, "q"]
  	& \widetilde{\mathrm{D}}^A \arrow[d, "q"] \\
  	\widetilde{\mathcal{D}}^I_F \arrow[r, dashed, hook]
  	& \widetilde{\mathcal{D}}^A_F.
  \end{tikzcd}
  \]
  In the above diagram, $q$ stands for the quotient map. It was shown in \cite[Prop. 6.2]{Amir_Sundar-product-system} that there exists a unique homomorphism $\lambda_1:\widetilde{\mathrm{D}}^{A} \to  \varinjlim_{s \in F}A_s$ such that, for $a \in A$,
  \[
   \lambda_1(j_s(a)) = \begin{cases}
   	\lambda_s(a) & \textup{ if } s\in F,\\
   	0 & \textup{ if } s\notin F.
   \end{cases} 
  \]
   Moreover, it was also shown that  $\lambda_1$ factors through to give an isomorphism, denoted $\widetilde{\lambda}_1$, between $\widetilde{\mathcal{D}}^A_F$ and the inductive limit $ \varinjlim_{s \in F}A_s$. 
  Since, for $x \in A$ and $g,g_1,g_2,\cdots,g_n \in G$,
  \[
  1_{\Omega g_1}1_{\Omega g_2}\cdots 1_{\Omega g_n} j_g(x) =j_{g_1}(1)j_{g_2}(1)\cdots j_{g_n}(1)j_g(x),
  \]
  we have 
  \begin{align*}
  \lambda_1( 1_{\Omega g_1}1_{\Omega g_2}\cdots 1_{\Omega g_n} j_g(x)) &= \lambda_1(j_{g_1}(1))\lambda_1(j_{g_2}(1))\cdots \lambda_1(j_{g_n}(1)) \lambda_1(j_{g}(x))\\
  &=  1_{\Omega g_1}(F)1_{\Omega g_2}(F)\cdots 1_{\Omega g_n}(F)  \lambda_1(j_{g}(x)).
\end{align*}
 Since \(C_0(\widetilde{\Omega})\) is the closure of  \(\textup{span}\{1_{\Omega g_1}1_{\Omega g_2}\cdots 1_{\Omega g_n}: g_1, g_2,\cdots, g_n\in G\}\), the above computation gives us
 \begin{equation}
     \label{master_one}
      \lambda_1(\phi j_g(x)) = \phi(F)\lambda_1(j_g(x))
 \end{equation}
 for $x \in A$. Hence, 
 \begin{equation*}
     \lambda_1(\phi d)=\phi(F)\lambda_1(d)
 \end{equation*}
 for $ \phi \in C_{0}(\widetilde{\Omega})$ and $d \in \widetilde{\mathrm{D}}^{A}$.  In particular, the above equation together with Eq.~\eqref{master_one} imply that $\lambda_1$ factors through $\widetilde{\mathcal{D}}^{I}_F$ and maps $\widetilde{\mathcal{D}}^{I}_F$ onto the $C^{*}$-algebra $\varinjlim_{s \in F}I_s \subset \varinjlim_{s \in F}A_s$. The resulting map thus obtained from $\widetilde{\mathcal{D}}^{I}_F \to  \varinjlim_{s \in F}I_s$ is denoted by $\widehat{\lambda}_1$.
 
 Since $\widetilde{\lambda}_1\colon \widetilde{\mathcal{D}}_F^{A} \to \varinjlim_{s \in F}A_s$ is an isomorphism and $\widetilde{\mathcal{D}}^I_F \hookrightarrow \widetilde{\mathcal{D}}^A_F$ is an embedding, it follows that  $\widehat{\lambda}_1$ is an isomorphism. This proves the claim. 
 
 Applying Lemma~\ref{lem-ind-limit-resp-short-ext} to Eq.~\eqref{eq-prop-exact-seq-1},  we obtain
\begin{equation*}
	0 \longrightarrow \widetilde{\mathcal{D}}_F^{I} \longrightarrow \widetilde{\mathcal{D}}_F^{A} \longrightarrow \widetilde{\mathcal{D}}_F^{B}\longrightarrow 0.
\end{equation*}
Recall that $\mathcal{D}^{\bullet}$ is the restriction of $\widetilde{\mathcal{D}}^{\bullet}$ onto $\Omega$. Thus, for $F \in \Omega$, we have a short exact sequence 
\begin{equation*}
	0 \longrightarrow \mathcal{D}_F^{I} \longrightarrow \mathcal{D}_F^{A} \longrightarrow \mathcal{D}_F^{B}\longrightarrow 0.
\end{equation*}
Applying Lemma \ref{lem-fiber-pres-ind-lim},  we obtain the following short exact sequence 
\[
0 \longrightarrow C(\Omega, \mathcal{D}^{I}) \longrightarrow C(\Omega, \mathcal{D}^{A}) \longrightarrow C(\Omega,\mathcal{D}^{B}) \longrightarrow 0.
\]
Since, the full groupoid crossed product functor is exact (see~\cite[Lemma 6.3.2]{Anantharaman} or~\cite[Thm. 18]{IW-2012-Ideal-st}), we have the following short exact sequence
\[
0 \longrightarrow C(\Omega, \mathcal{D}^{I})\rtimes \mathcal{G} \longrightarrow C(\Omega,{\mathcal{D}}^{A})\rtimes \mathcal{G} \longrightarrow C(\Omega, \mathcal{D}^B)\rtimes \mathcal{G} \longrightarrow 0.
\]
Since \(\mathcal{G}\) is amenable, we have the following short exact sequence
\[
0 \longrightarrow C(\Omega, \mathcal{D}^{I})\rtimes_{\red} \mathcal{G} \longrightarrow C(\Omega,{\mathcal{D}}^{A})\rtimes_{\red} \mathcal{G} \longrightarrow C(\Omega, \mathcal{D}^B)\rtimes_{\red} \mathcal{G} \longrightarrow 0.
\]
Since \((P,G)\) satisfies the Toeplitz condition,~\cite[Thm. 4.3]{Sundar_Khoshkam} allows us to rewrite the last equation as
\[
0 \longrightarrow I\rtimes_{\red} P \longrightarrow A\rtimes_{\red} P \longrightarrow B \rtimes_{\red} P\longrightarrow 0.
\]
This completes the proof.
\end{proof}

\begin{remark}
\label{inductive_limit_non-unital_remark}
   Let $(A,P,\alpha)$ be a semigroup dynamical system, and let $(\mathcal{D},\mathcal{G})$ be the associated groupoid dynamical system. We see from the proof of Prop.~\ref{prop-exact-func}, by applying to the dynamical systems $(A, P,\alpha)$ and its unitisation, that the fibre $\mathcal{D}_F$ is isomorphic to the inductive limit $ \varinjlim_{s \in F}A_s$ provided $(P, G)$ satisfies the hypothesis of Prop.~\ref{prop-exact-func}. 
\end{remark}

\begin{remark}
LaLonde (\cite[P. 242]{Lalonde}) defined the notion of \emph{separable exactness} for a $C^{*}$-algebras as follows: A $C^{*}$-algebra $A$ is said to be separably exact if for every short exact sequence 
\[
0 \longrightarrow I \longrightarrow B \longrightarrow C \longrightarrow 0\]
of separable $C^{*}$-algebras, the sequence 
\[
0 \longrightarrow I\otimes A  \longrightarrow B \otimes A \longrightarrow C \otimes A \longrightarrow 0\]
is short exact. He further  showed that for a $C^{*}$-algebra $A$, $A$ is exact if and only if $A$ is separably exact (\cite[Prop. 6.13]{Lalonde}). 
\end{remark}

The following lemma will be used in the proof of Thm.~\ref{thm-exactness}.  Although its proof is routine and is probably  known to experts, we were unable to find a suitable reference in the literature. Thus, we include a proof for completeness.
\begin{lemma}\label{lem-unital-exact}
	Let \(A\) be a \(C^*\)-algebra. Suppose that for every short exact sequence 
	\[
	0\longrightarrow I\longrightarrow B\longrightarrow C\longrightarrow0,
	\]
	in which \(B\) and \(C\) are unital separable $C^{*}$-algebras, the following sequence
	\[
	0\longrightarrow I\otimes A
	\longrightarrow B\otimes A
	\longrightarrow C\otimes A
	\longrightarrow 0
	\]
	is exact. Then, \(A\) is separably exact.
\end{lemma}

\begin{proof}
	Let
	\[
	0\longrightarrow I
	\overset{\iota}{\longrightarrow}
	B
	\overset{\pi}{\longrightarrow}
	C
	\longrightarrow0
	\]
	be an arbitrary short exact sequence of separable \(C^*\)-algebras. After unitizations, we obtain the short exact sequence
	\[
	0\longrightarrow I
	\longrightarrow B^{+}
	\overset{\pi^{+}}{\longrightarrow}
	C^{+}
	\longrightarrow0,
	\]
	where \(\pi^{+}(b,\lambda)=(\pi(b),\lambda)\) for \((b,\lambda) \in B^+\). Also, we have a commutative diagram with exact columns
	\[
	\begin{tikzcd}[column sep=2.8em,row sep=2.2em]
			0 \arrow[r] & 0 \arrow[r]  & 0 \arrow[r] & 0 \arrow[r]  & 0\\
		0 \arrow[r] & 0 \arrow[r] \arrow[u] & \mathbb{C} \arrow[r,"\mathrm{id}"] \arrow[u] & \mathbb{C} \arrow[r] \arrow[u] & 0 \\
		0 \arrow[r] & I \arrow[r] \arrow[d,equal]  \arrow[u]
		& B^{+} \arrow[r,"\pi^{+}"] \arrow[u] & C^{+} \arrow[r] \arrow[u] & 0 \\
		0 \arrow[r] & I \arrow[r] & B \arrow[r,"\pi"] \arrow[u, hook] & C \arrow[r] \arrow[u, hook] & 0 \\
		0 \arrow[r] & 0 \arrow[r] \arrow[u] & 0 \arrow[r] \arrow[u] & 0 \arrow[r] \arrow[u] & 0.
	\end{tikzcd}
	\]
	Tensoring with \(A\) gives the following commutative diagram
	\[
	\begin{tikzcd}[column sep=2.6em,row sep=2.4em]
		0 \arrow[r] & 0 \arrow[r] & 0 \arrow[r]  & 0 \arrow[r] & 0\\
		0 \arrow[r] & 0 \arrow[r] \arrow[u] & A \arrow[r,"\mathrm{id}"] \arrow[u]
		& A \arrow[r] \arrow[u] & 0 \\
		0 \arrow[r] & I\otimes A \arrow[r] \arrow[d,equal] \arrow[u] & B^{+}\otimes A \arrow[r,"\pi^{+}\otimes\mathrm{id}"] \arrow[u] 
		& C^{+}\otimes A \arrow[r] \arrow[u]
		& 0 \\
		0 \arrow[r] & I\otimes A \arrow[r] 
		& B\otimes A \arrow[r,"\pi\otimes\mathrm{id}"] \arrow[u, hook] & C\otimes A \arrow[r] \arrow[u, hook] & 0 \\
		0 \arrow[r] & 0 \arrow[r] \arrow[u] & 0 \arrow[r] \arrow[u] & 0 \arrow[r] \arrow[u] & 0.
	\end{tikzcd}
	\]
By the hypothesis, the third row, the third column and the fourth columns are exact. The other columns are exact is tautological. By  diagram chasing, we can conclude that the fourth row is also exact.
	Hence, \(A\) is separably exact. 
\end{proof}

\begin{theorem}\label{thm-exactness}
	Suppose \((A, P, \alpha)\) is a semigroup dynamical system. We assume that $A$ is separable. Suppose that \((P,G)\) satisfies the Toeplitz condition, every element of \(\widetilde{\Omega}\) is directed, and that the Wiener--Hopf groupoid \(\mathcal{G}\) is amenable. Then, \(A\rtimes_{\red}P\) is exact if and only if \(A\) is exact.
\end{theorem}
\begin{proof}
The `only if' part follows from the fact that  $A$ is a subalgebra of $A \rtimes_{\red} P$ and exactness passes to subalgebras. For the `if part', suppose that $A$ is exact. 
	Let $A^{+}$ be the unitisation of $A$. Then, $A^{+}$ is exact. For $s \in P$, let $\alpha_{s}^{+}:A^{+} \to A^{+}$ be defined by 
	$\alpha_{s}(x,\lambda)=(\alpha_s(x),\lambda)$ for $(x,\lambda) \in A^{+}=A \oplus \mathbb{C}$. It follows from Corollary 2.7 of~\cite{Amir_Sundar-product-system} that $A \rtimes_{\red} P$ is an ideal of $A^{+}\rtimes_{\red} P$. As exactness passes to ideals and subalgebras, it suffices to prove that $A^{+}\rtimes_{\red} P$ is exact. In other words, to prove the `if part' we can assume that $A$ is unital and $\alpha_s$ is unital, i.e., $\alpha_s(1)=1$ for each $s \in P$. 
 
Thus, assume that \(A\) is exact, unital and $\alpha_s$ is unital for every $s \in P$. By~\cite[Prop. 6.13]{Lalonde}, it suffices to prove that $A \rtimes_{\red} P$ is separably exact.
Let 
	\begin{equation*}
	0 \longrightarrow I \longrightarrow B \longrightarrow C \longrightarrow 0
\end{equation*}
be a short exact sequence of separable \(C^*\)-algebras.  Thanks to Lemma~\ref{lem-unital-exact}, we may further assume that $B$ and $C$ are unital. 
  Since \(A\) is exact, we have the following short exact sequence
\[
 	0 \longrightarrow I\otimes A \longrightarrow B\otimes A \longrightarrow C\otimes A \longrightarrow 0.
\]
Applying Prop.~\ref{prop-exact-func} to the above sequence (with the trivial action of \(P\) on \(I, B, C\) and the given action on \(A\)), we obtain the following short exact sequence
\[
	0 \longrightarrow (I\otimes A)\rtimes_{\red}P \longrightarrow (B\otimes A)\rtimes_{\red}P \longrightarrow (C\otimes A)\rtimes_{\red}P \longrightarrow 0.
\]
Using Lemma~\ref{lem-comp-tens-prod}, we can rewrite the last equation as
\begin{equation*}
	0 \longrightarrow I\otimes (A\rtimes_{\red}P) \longrightarrow B\otimes (A\rtimes_{\red}P) \longrightarrow C\otimes (A\rtimes_{\red}P) \longrightarrow 0.
\end{equation*}
Therefore, \(A\rtimes_{\red}P\) is separably exact.  Hence, by ~\cite[Prop. 6.13]{Lalonde}, $A \rtimes_{\red} P$ is exact. 
\end{proof}

\begin{corollary}\label{main-coro-exactness}
Let \(P\) be a subsemigroup of \(G\). Suppose that $P$ has an order unit, i.e., there exists $a \in P$ such that  \(\bigcup_{n\in \N}Pa^{-n} = G\). Let $B$ be a $C^{*}$-algebra, and let $X$ be a product system of $B$-$B$-correspondences over $P^{\op}$. Assume that each fibre \(X_s\) is full and the left action of $B$ on $X_s$ is by compact operators. Suppose that every element of \(\Omega\) is directed and that the Wiener--Hopf groupoid \(\mathcal{G}\) is amenable. Then \(C^*_{\red}(X)\) is exact if and only if the coefficient algebra \(B\) is exact.
\end{corollary}
\begin{proof}
Observe that right translation preserves the pre-order $\leq$. Thus, a right translate of a directed set is directed. Hence,  every element of \(\widetilde{\Omega}\) is directed.

We claim that $(P,G)$ satisfies the Toeplitz condition. 
	Let $\rho$ be the right
	regular representation of G on $\ell^2(G)$. For $g \in G$, 
	let $w_g$ be the compression of $\rho(g)$ to $\ell^2(P)$.
	For $s \in P$, let $v_s:=w_s$. Let $\{\delta_s\}_{s \in P}$
	be the standard orthonormal basis for $\ell^2(P)$.
	Then, for $s \in P$ and $g \in G$, 
	\begin{equation}\label{eq-toeplitz-cond-proof}
		v_s(\delta_t)=\delta_{ts}; \quad w_g(\delta_t):= \begin{cases}
			\delta_{tg} & \textup{ if } tg \in P,\\
			0 & \textup{ otherwise.}
		\end{cases} 
	\end{equation} 
	Let $g \in G$ be given. Since $G=\bigcup_{n=1}^{\infty}Pa^{-n}$, there exists 
	$n \geq 1$ and $x \in P$ such that $g=xa^{-n}$. A direct check with
	Eq.~\eqref{eq-toeplitz-cond-proof} shows that $w_g=v_{a^n}^{*}v_x$. Hence, $(P,G)$ satisfies
	the Toeplitz condition.

Since \(P\) has an order unit,~\cite[Thm. 1.2]{Amir_Sundar-product-system} ensures  the existence of a countably generated Hilbert \(B\)-module \(\mathcal{E}\) and an \(E_0\)-semigroup \(\alpha\) over \(P\) on \(\mathcal{L}_B(\mathcal{E})\) such that \(X\) is isomorphic to the product system associated to~\(\alpha\).
Therefore, Thm.~\ref{thm-exactness} ensures that
 \(\mathcal{K}_B(\mathcal{E})\rtimes_{\red}P\) is exact if and only if \(\mathcal{K}_B(\mathcal{E})\) is exact.
  On the other hand, by~\cite[Thm. 1.1]{Amir_Sundar-product-system}, the reduced \(C^*\)-algebra \(C^*_{\red}(X)\) and \(\mathcal{K}_B(\mathcal{E})\rtimes_{\red} P\) are Morita equivalent. 
 Since exactness is preserved under Morita equivalence, \(C^*_{\red}(X)\) is exact if and only if \(\mathcal{K}_B(\mathcal{E})\) is exact. Since \(\mathcal{E}\) is full, \(\mathcal{K}_B(\mathcal{E})\) and \(B\) are Morita equivalent. Thus, \(\mathcal{K}_B(\mathcal{E})\) is exact if and only if \(B\) is exact. Consequently, \(C^*_{\red}(X)\) is exact if and only if \(B\) is exact.
\end{proof}

\begin{remark}
    We refer the reader to Remark 4.5 of \cite{Amir_Sundar-product-system} for a list of examples for which Thm. \ref{thm-exactness} and Corollary \ref{main-coro-exactness} are applicable. 
    We mention here that the conclusion of  Corollary \ref{main-coro-exactness} also holds if $P=\mathbb{F}_{n}^{+}$; the free semigroup on $n$ letters. For, in this case, every product system comes from an $E_0$-semigroup, and the Wiener--Hopf groupoid satisfies the amenability and the directedness hypotheses (see \cite[Section 6]{Sundar_Khoshkam} and \cite[Section 8.2]{Xin-Li-2012-Nucl-semigroup-alg-amen}). Then, the rest of the proof works as it is. 
\end{remark}

\section{Equivalence of Fell bundles}
\label{sec-equiv-Fell-bund}
In this section, we compare our groupoid crossed product picture with the Fell bundle constructed in~\cite{Rennie_Sims}. Our main result is that these two constructions give rise to Fell bundles that are equivalent in the sense of Muhly--Williams~(\cite{Muhly-Williams-Disintegration-theorem}).
We refer to~\cite{Muhly-Williams-Disintegration-theorem} for the basics of Fell bundles and upper semicontinuous Banach bundles.  For an upper semicontinuous Banach bundle $\E$, we denote the space of compactly supported continuous sections with the inductive limit topology by $\Gamma_{\E}$.  We implicitly assume that the bundles that we consider have \emph{enough sections}.

Let $\G$ be a locally compact Hausdorff \etale\ groupoid. 
Let \(p\colon \mathcal{B}\to \G\) be a Fell bundle, and let  \(q\colon\ \mathcal{E} \to \G\) be an upper semicontinuous Banach bundle. We say that \(\mathcal{B}\) acts on \(\mathcal{E}\) from the left if there is a  bilinear map \((b,f)\mapsto b\cdot f\) from \(\mathcal{B}*\mathcal{E}:= \{(b,f) \in \mathcal{B}\times \mathcal{E}: s(p(b)) = r(q(f))\}\) to \(\mathcal{E}\) such that 
\begin{enumerate}
	\item \(q(b\cdot f) =p(b)q(f)\);
	\item \((a\cdot b)\cdot f =a\cdot (b\cdot f)\) whenever they make sense, and
	\item \(\|b\cdot f\| \leq \|b\|\|f\|\).
\end{enumerate}
One can define a right action of \(\B\) on \(\E\) analogously. 
\begin{lemma}\label{lem-equi-cond-continuity}
Let \(\G\) be a locally compact Hausdorff \etale\ groupoid, let \(p\colon \B\to \G\) be a Fell bundle and let \(q\colon \E\to \G\) be an upper semicontinuous Banach bundle.  Suppose that $\B$ acts on $\E$ from the left. Then, the following statements are equivalent. 
\begin{enumerate}
	\item The map \(\B * \E \ni (b,e)\mapsto b\cdot e \in \E\)  is continuous.
	\item\label{cond-2-left-action} For sections \(f\in \Gamma_{\B}\) and \(g\in \Gamma_{\E}\), the convolution \(f*g\in \Gamma_{\E}\).
\end{enumerate}
\end{lemma}
\noindent Here, the convolution \(f*g\) is defined by 
\[
 f*g(\gamma) =\sum_{\eta\zeta=\gamma}f(\eta)g(\zeta)
\]
for \(f\in \Gamma_{\B}, g\in \Gamma_{\E}\) and \(\gamma \in \G\).

\begin{proof}[Proof of Lemma~\ref{lem-equi-cond-continuity}]
\noindent (2).\(\implies\)(1). Let \((b_i,e_i)\in \B*\E\) be a net such that \((b_i,e_i) \to (b,e)\). Let \(\gamma_i:=p(b_i), \eta_i:= q(e_i), \gamma:=p(b), \textrm{~and~}\eta:=q(e)\). Then, \(\gamma_i\to \gamma\) and \(\eta_i\to \eta\) in~\(\G\).  Choose open bisections \(U\) and \(V\) of \(\G\) containing \(\gamma\) and \(\eta\), respectively, such that eventually \(\gamma_i \in U\) and \(\eta_i \in V\). Since the bundles \(\B\) and \(\E\) have \emph{enough sections}, there exist sections \(f\in \Gamma_{\B}\) and \(g\in \Gamma_{\E}\) with \(\textup{supp}(f)\subset U\) and \(\textup{supp}(g)\subset V\) such that \(f(\gamma) = b\) and \(g(\eta)=e\). Condition~\eqref{cond-2-left-action} ensures that \(f*g \in \Gamma_{\E} \). Since both \(f\) and \(g\) are bisection supported, we have 
	\[
	 f*g(\gamma_i\eta_i) = f(\gamma_i)g(\eta_i) \quad \textup{and} \quad
     f*g(\gamma \eta)=f(\gamma)g(\eta).\]
     Since $f*g \in \Gamma_{\E}$ (by assumption), $f*g(\gamma_i\eta_i) \to f*g(\gamma \eta)=be$.
     Hence, 
     \begin{equation}
         \label{limitu}
         \lim_{i}f(\gamma_i)g(\eta_i)=be.
     \end{equation}
    Note that 
	\[\	\|f(\gamma_i)g(\eta_i)-b_ie_i\| \leq \|f(\gamma_i)-b_i\|\|g(\eta_i)\| + \|b_i\|\|g(\eta_i)-e_i\|.
	\]
Since $b_i \to b$, $f(\gamma_i) \to b$, $e_i \to e$ and $g(\eta_i) \to e$, we can conclude from the above inequality that 
\begin{equation}
    \label{limitu1}
    \lim_{i}\|f(\gamma_i)g(\eta_i)-b_ie_i\|= 0.
\end{equation}
Eq.~\eqref{limitu} and Eq.~\eqref{limitu1} implies that $b_ie_i\to be$ (see \cite[Prop. C.20]{Dana-Williams}\footnote{The proof given for upper semicontinuous bundles of $C^{*}$-algebras works for Banach bundles too.}).  Hence, the map \(  (b,e) \mapsto be\) is continuous. 

\noindent (1).\(\implies\)(2). This is standard, and hence omitted. 
\end{proof}

 We now recall the notion of equivalence of Fell bundles in a special case, i.e., the underlying groupoids are the same for both the Fell bundles. We refer the reader to~\cite[Definition 6.1]{Muhly-Williams-Disintegration-theorem} for the general definition.
 
Let $\mathcal{G}$ be a locally compact Hausdorff \etale\  groupoid, let \(p_{\A}\colon \A\to \G\) and \(p_{\B}\colon \B \to \G\) be two Fell bundles over \(\G\). Let \(q\colon\E \to \G \) be an upper semicontinuous Banach bundle. Let \begin{align*}
   \E*_s\E &:=\{(e,f) \in \E \times \E: s(q(e)) = s(q(f))\};\\
   \E*_r\E&:=\{(e,f) \in \E \times \E: r(q(e))=r(q(f))\}.
\end{align*}
We say \(\E\) is an \emph{equivalence} between the Fell bundles \(\A\) and \(\B\) if there is a continuous left \(\A\)-action and a continuous right \(\B\)-action on \(\E\) which commute, and there are sesquilinear forms \(\E*_s\E \ni (e,f) \mapsto \Linpro{e}{f}\in \A\) and \(\E*_r\E \ni (e,f) \mapsto \inpro{e}{f} \in \mathcal{B}\) satisfying the usual positivity conditions on the unit fibres such that the following conditions hold. 
\begin{enumerate}
      \item \(p_{\A}(\Linpro{e}{f})=q(e)q(f)^{-1}\) and \(p_{\B}(\inpro{e}{f})=q(e)^{-1}q(f)\)
	\item \(\Linpro{e}{f}^*=\Linpro{f}{e}\) and \(\inpro{e}{f}^*=\inpro{f}{e}\);
	\item \(\Linpro{a\cdot e}{f} = a\Linpro{e}{f}\), \(\inpro{e}{f\cdot b} = \inpro{e}{f}b\) for  $a \in \A$ and $b\in \B$;
	\item  \(\Linpro{e}{f}g =e\inpro{f}{g}\).
\end{enumerate}
Also with the above inner products and actions, for \(\gamma\in \G\), \(\E_{\gamma}\) is a \(\A_{r(\gamma)}\)-\(\B_{s(\gamma)}\) imprimitivity bimodule.

We now show that the Fell bundle introduced in~\cite{Rennie_Sims} and the Fell bundle associated to our groupoid dynamical system are equivalent. Let $P$ be a subsemigroup of a countable discrete group $G$, and let $X$ be a product system of \(B\)-\(B\)-correspondences over $P^{\op}$. 
Note that to make sense of both Fell bundles, we need the following assumptions:
\begin{enumerate}
	\item[(\textbf{C1})] \((P, G)\) is quasi-lattice ordered, i.e., given two elements $x,y$, either they have no upper bound, or they have a least upper bound w.r.t. the partial order $\leq$. For $x,y \in G$, if $x$ and $y$ have an upper bound, we denote their l.u.b. by $x \vee y$; otherwise, we set $x \vee y=\infty$. 
	\item[(\textbf{C2})] \(X=\{X_s\}_{s\in P}\) is a  proper product system, and $X_s$ is full for every $s \in P$;
	\item[(\textbf{C3})] the product system \(X\) comes from an \(E_0\)-semigroup; i.e., there exists a full Hilbert $B$-module  \(E\) and unitaries  \(\sigma_s\colon E\otimes_BX_s\to E \) such that  
	\[
	 \sigma_t(\sigma_s\otimes 1) =\sigma_{ts}(1\otimes U_{s, t})
	\]
	for \(s,t\in P\). Here, \(U_{s,t}\colon X_s\otimes_BX_t\to X_{ts}\) denotes the unitary defining the multiplication of the product system \(X\).
	\end{enumerate}
Let the notation be as above for the rest of this paper. 

For $s \in P$, let $\alpha_s:\K_{B}(E) \to \K_{B}(E)$ be defined by $\alpha_s(T)=\sigma_s(T \otimes 1)\sigma_{s}^{*}$. Then, $\alpha:=\{\alpha_s\}_{s \in P}$ is a semigroup of endomorphisms of $\K_B(E)$. Let $(\D,\G)$ be the groupoid dynamical system associated with the semigroup dynamical system $(\K_B(E),P,\alpha)$. We denote the action of $\G$ on $\D$  by  $\widetilde{\alpha}:=\{\widetilde{\alpha}_{(F,g)}\}_{(F,g) \in \G}$. Since we have assumed that $(P,G)$ is quasi-lattice ordered, every element of $\Omega$ is directed (\cite[Lemma 6.4]{Amir_Sundar-product-system}), and then by Remark \ref{inductive_limit_non-unital_remark}  the fibre $\D_F$ can be identified with $\varinjlim_{s \in F}\K_B(E)$ with the connecting maps given by $\alpha_{ts^{-1}}$ if $s \leq t$. For $F \in \Omega$ and $p \in P$, let  $\lambda_{p}^F:\K_{B}(E) \to \varinjlim_{s \in F}\K_B(E)=\B_F$ be the canonical map. Then, for $(F,g) \in \G$, $\widetilde{\alpha}_{(F,g)}:\B_{F.g} \to \B_{F}$ satisfies the equality
\[
\widetilde{\alpha}_{(F,g)}(\lambda_{pg}^{Fg}(T))=\lambda_p^{F}(T)
\]
for $p \in F$ and $T \in \K_B(E)$. 
Let $\B$ be the Fell bundle over $\G$ associated to the groupoid dynamical system $(\D,\G,\widetilde{\alpha})$ (see ~\cite[Example 2.1]{Muhly-Williams-Disintegration-theorem}).

Let us next recall the Fell bundle considered in \cite{Rennie_Sims}.  We denote it by $\A$.  Let $(F,g) \in \G$, and let \[
d_{(F,g)}:=\{r \in F: r \geq e \vee g^{-1}\}.\]

\emph{Claim:} Let $(F,g) \in \mathcal{G}$. Then,  $e \vee  g^{-1}<\infty$ and $d_{(F,g)}$ is non-empty.

\noindent To see this, choose a sequence $(a_n)_n \in P$ such that $(P^{-1}a_n)_n \to F$. Hence, \[1_F(e)=\lim_{n \to \infty}1_{P^{-1}a_n}(e)=1.\] 
Hence, for $F \in \Omega$, $e \in F$. Since $(F,g) \in \mathcal{G}$, $Fg \in \Omega$.  For the same reason, $e \in Fg$, i.e., $g^{-1} \in F$. Since $F$ is directed, there exists $r \in F$ such that $r \geq e$ and $r \geq g^{-1}$. Hence, $e \vee g^{-1}<\infty$ and $d_{(F,g)}$ is non-empty.

  Moreover, as  $F$ is directed, it follows that $d_{(F,g)}$ is a directed set. The fibre $\A_{(F,g)}$ is given by
\[
\A_{(F,g)}:=\varinjlim_{p \in d_{(F,g)}}\K_B(X_
{pg}, X_{p}).
\]
The connecting maps $\K_B(X_{pg},X_{p}) \to \K_B(X_{apg},X_{ap})=\K_B(X_{pg}\otimes X_{a},X_{p} \otimes X_a)$ are given by $T \mapsto T \otimes 1$. The product rule on $\A$ is given by `composition', i.e., we can choose representatives in such a way that the composition makes sense, and it can be verified that the resulting rule is independent of all the choices. The $*$-operation is given by taking adjoints.  For $p,q \in P$ and $S \in \K_B(X_q,X_p)$, define a section $\phi_{p,q,S}:\G \to \A$ by 
\[
\phi_{p,q,S}(F,g) = \begin{cases}
		[S] & \textup{ if } g=p^{-1}q,\\
	0 & \textup{ otherwise. } 
	\end{cases}
\]
Then, $\A$ carries a unique topology which makes it an upper semicontinuous Fell bundle over $\G$ such that  $\{ \phi_{p,q,S}:p,q \in P,S \in \K_B(X_q,X_p)\}$ is total in~$\Gamma_{\A}$.

 Next, we describe an upper semicontinuous Banach bundle $\E$ over $\G$ as follows:   for \((F,g) \in \G\), define 
\[
 \E_{(F,g)} = \varinjlim_{p\in d_{(F,g)}} \K_B(E, X_p).
\]
Here, the connecting map
\(\K_B(E, X_p)\to  \K_B(E, X_{q})\), for $p \leq q$, is given by \(T\mapsto (T\otimes 1_{qp^{-1}})\sigma_{qp^{-1}}^{-1}\). We remind the reader that $\sigma:=\{\sigma_s\}_{s \in P}$ are the unitaries mentioned in Condition $\textbf{(C3)}$. 

 The left action of $\A$ and the right action of $\B$ on $\E$ are given by composition. To give a little bit of detail, let us explain the left action of $\A$. The formulae that appear below are only densely defined, and one can show that they extend by continuity.
 
Let \((F,g)\) and \((F^\prime,g^\prime)\) be two composable elements of \(\G\). Then, $F^{\prime}=Fg$.  Let \(a\in \A_{(F,g)}\) and \(e\in \E_{(F^\prime,g^\prime)}\). Choose representatives \(S\in \K_B(X_{pg}, X_{p})\) and \(T\in \K_B(E, X_q)\) such that \(a=[S]\) and \(e=[T]\) for some \(p \in d_{(F,g)}\) and \(q \in d_{(F^\prime,g^\prime)}\). 

Since $Fg$ is directed and hereditary, $q_0:=pg \vee q \in Fg$. Let $q_0=s_0pg$ and $q_0=s_1q$ for some $s_0,s_1\in P$. Then,   \begin{equation}
    \label{left action definition formula}
    a\cdot e:=[(S \otimes 1_{s_0})((T\otimes 1_{s_1})\sigma_{s_1}^{-1})]\in \K_B(E, X_{s_0p}).
\end{equation} In the above, $1_{s_i}$ is the identity operator on $X_{s_i}$.
It can be checked that this is a well-defined left action of \(\A\) on \(\E\). Similarly, we can define a right action of \(\B\) on \(\E\). For inner products, we take \(e= [S]\in \E_{(F,gg^{\prime})}\) with $S \in \K_B(E,X_p)$ and \(f=[T]\in \E_{(Fg,g^{'})}\) with $T \in \K_B(E,X_{pg})$, and define 
\[
 \Linpro{e}{f} :=[ST^*]\in \K_B(X_{pg},X_p)\subset \A_{(F,g)}.
\] 
The right inner product has a similar expression given by  
\[
\inpro{[S]}{[T]} :=[S^*T]\in \K_B(E) \subset \B
\]
for appropriate \(S\) and \(T\).
The representatives can be chosen so that the compositions $S^{*}T$ and $ST^{*}$ are well-defined. 
All the algebraic properties of the Fell bundle equivalence are straightforward to verify, so we omit them. 

To make $\E$ an upper semicontinuous bundle we need the following lemma.
\begin{lemma}\label{lem-cont-of-bund}
	Let \(p_0\in P\) be fixed, and let \([p_0, P):=\{p\in P: p\geq p_0\}\). Suppose \(f\colon [p_0, P) \to [0, \infty)\)
is a decreasing function, i.e., $f(x) \leq f(y)$ whenever $x \geq y$. Then, the map \(\phi\colon \Omega \to [0, \infty)\) defined by 
\[
\phi(F) = \begin{cases}
	0 & \textup{ if } p_0\notin F,\\
	\displaystyle \inf_{p\in F, p \geq p_0}f(p) & \textup{ if } p_0\in F
\end{cases}
\] 
is upper semicontinuous.
\end{lemma}
\begin{proof}
Let \(\alpha>0\), and let \(R= \{F\in \Omega : \phi(F)<\alpha\}\). Fix \(F_0\in R\).

\noindent Case (I): \(p_0\notin F_0\). Then, \(F_0\in \{F\in \Omega:p_0\notin F\} \subseteq R\). As $\{F \in \Omega: p_0 \notin F\}$ is open, it follows that  $F_0$ is an interior point of  \(R\).

\noindent Case (II): \(p_0\in F_0\). Since \(\phi(F_0) = \inf\{f(p): p\in F_0 \textup{ and } p\geq p_0\} <\alpha\), there exists \(p_1\in F_0\) such that \(f(p_1)<\alpha\). Consider the basic open set \(U_{p_1} = \{F \in \Omega : p_1\in F\}\). For any
\(F\in U_{p_1}\), we have \(\phi(F)\leq f(p_1)<\alpha\).
Hence, \(F_0\in U_{p_1}\subseteq R\), and in this case too, $F_0$ is an interior point of~$R$.
\end{proof}

 For $\phi \in C(\Omega)$ and $g \in G$, let $\phi \otimes e_g \in C_{c}(\G)$ be the function defined by \[
\phi\otimes e_g(F,h) = \begin{cases}
		\phi(F) & \textup{ if  $h=g$},\\
	0 & \textup{ otherwise. } 
	\end{cases}
\]
For $(F,g) \in \G$ and for $p \in d_{(F,g)}$, denote the natural map $\K_{B}(E,X_p) \to  \varinjlim_{q \in d_{(F,g)}}\K_{B}(E,X_q)$ by $i_{p}^{(F,g)}$.  
For \(p\in P\), \(T\in \K_{B}(E,X_p)\) and $g \in G$, define a  section $f_{p,g,T}$ of the bundle $\E$ by
\[
f_{p,g, T}(F,h) = \begin{cases}
		i_{p}^{(F,g)}(T) & \textup{ if }  p\in d_{(F,g)} ~~\textup{and $h=g$},\\
	0 & \textup{ otherwise. } 
	\end{cases}
\]
Let $\Gamma:=\overline{\textup{span}\{f_{p,g,T}: p\in P, T \in \K_B(E,X_p),g \in G\}}$. 
Here, the closure is taken with respect to the inductive limit topology.

 The following lemma is similar to the one given in \cite[Lemma 3.4]{Rennie_Sims}. However, the proof given there is not very clear to the authors, and the analogous statement of~\eqref{lem-cont-cond-1} is not to be found in \cite{Rennie_Sims}, which is crucial to justify the surjectivity part in \cite[Thm. 5.1]{Rennie_Sims}. Also, the verification that the multiplication operation on $\A$ is continuous seems to be omitted.  
For these reasons, the authors have decided to include proofs of the next two lemmas. 

\begin{lemma}
With the foregoing notation, we have the following.
\begin{enumerate}
    \item\label{lem-cont-cond-1} For $\phi \in C_{c}(\G)$ and $f \in \Gamma$, the pointwise product $\phi f \in \Gamma$.
    \item\label{lem-cont-cond-2} For $f \in \Gamma$, the map \[
    \Gamma \ni (F,g)  \mapsto \|f(F,g)\| \in \mathbb{R}
    \]
    is upper semicontinuous.
    \item\label{lem-cont-cond-3} For every $(F,g) \in \G$, $\{f(F,g):f \in \Gamma\}$ is dense in $\E_{(F,g)}$. 
\end{enumerate}
Also, there exists a unique topology on \(\E\) that makes \(\E\) an upper semicontinuous Banach bundle such that $\Gamma=\Gamma_{\E}$.
\end{lemma}
\begin{proof}
Note that $\textup{span}\{1_{\Omega p} \otimes e_g:g \in G, p \in P\}$ is dense in $C_{c}(\G)$ with the inductive limit topology. To see this, first observe that as $P$ is quasi-lattice ordered, $\Omega p \cap \Omega q$ is empty if $p \vee q=\infty$ and $\Omega p \cap \Omega q=\Omega (p \vee q)$ if $p \vee q<\infty$. Also, $\{1_{\Omega p}: p \in P\}$ separates points of $\Omega$. Hence, $\textup{span}\{1_{\Omega p}: p \in P\}$ is an algebra and is dense in $C(\Omega)$. Consequently, $\{1_{\Omega p} \otimes e_g: p \in P, g \in G\}$ is total in $C_{c}(\G)$.

Let $p,q \in P$, $T \in \K_{B}(E,X_p)$ and $g_1,g_2 \in G$ be given. Suppose $p$ and $q$ have an upper bound. Let $s=(p \vee q)p^{-1}$. Then,  
\begin{equation}
    \label{master_equation_upper_semi}
    (1_{\Omega q} \otimes e_{g_1})f_{p,g_2, T} = \begin{cases}
		f_{p \vee q, g_1, (T\otimes 1_{s})\sigma_{s}^{-1}} & \textup{ if } g_1=g_2,\\
	0 & \textup{ if $g_1 \neq g_2$ }. 
	\end{cases}
    \end{equation}
    If $p$ and $q$ do not have an upper bound, then
    \begin{equation}
        \label{master_equation_upper_semi_1}
        (1_{\Omega q}\otimes e_{g_1})f_{p,g_2,T}=0.
    \end{equation}
As $\{1_{\Omega p} \otimes e_g: p\in P, g \in G\}$ is total in $C_{c}(\G)$, it follows from Eq.~\eqref{master_equation_upper_semi} and Eq.~\eqref{master_equation_upper_semi_1} that for $\phi \in C_{c}(\G)$ and $f \in \Gamma$, $\phi f \in \Gamma$. This proves $(1)$. 

Next, we prove $(2)$. It suffices to prove that $\G \ni (F,g) \mapsto \|f(F,g)\| \in \mathbb{R}$ is upper semicontinuous when $f$  is of the form
\[
f=\sum_{i=1}^{n}f_{p_i,g_i,T_i}.
\]
Let $ f:=\sum_{i=1}^{n}f_{p_i,g_i,T_i}$ be one such section. It suffices to prove that  for every $j \in \{1,2,\cdots,n\}$, the map
\[
\Omega \ni F \mapsto \|f(F,g_j)\|=\Big\|\sum_{g_i=g_j, p_i \geq g_{i}^{-1}}f_{p_i,g_i,T_i}(F,g_j)\Big\| \in \mathbb{R}
\]
is upper semicontinuous. Hence, we can, without loss of generality, assume that $g_i=g_j$ for all $i,j$, and set $g:=g_i$. And we can assume that $p_i \geq g^{-1}$ for all $i$. Moreover, 
$f_{p,g,T}+f_{p,g,S}=f_{p,g,T+S}$, and hence we can further assume that $p_i$'s are distinct. 

Let $\phi:\Omega \to \mathbb{C}$ be defined by $\phi(F):=\|f(F,g)\|$. We need to show that $\phi$ is upper semicontinuous. 
Let $J \subset \{1,2,\cdots,n\}$, and let \[U_{J}:=\{F \in \Omega: p_i \in F \textrm{~~for $i \in J$ and } p_i \notin F \textrm{~for $i \notin J$}\}.\]
        
        Suppose $J \neq \emptyset$ and  $U_J \neq \emptyset$. As every element of $\Omega$ is directed and hereditary, $p_J:=\bigvee_{j \in J}p_j<\infty$.
      Moreover, if $F\in U_J$, then $p_j\in F$ for every $j\in J$. As $F\in \Omega$ is directed, it follows that $p_J\in F$.
    For $j \in J$, set $s_j:=p_{J}p_j^{-1}$.  Set $S:=\sum_{j \in J}(T_j \otimes 1_{s_j})\sigma_{s_j}^{-1} \in \K_B(E,X_{p_J})$. Define a function $\chi_0:(p_J,P] \to [0,\infty)$ by \[
    \chi_0(p):=\|(S \otimes 1_{pp_{J}^{-1}})\sigma_{pp_J^{-1}}^{-1}\|.
    \] Then, $\chi_0$ is decreasing. Define $\chi:\Omega \to \mathbb{C}$ by\[
\chi(F) = \begin{cases}
	0 & \textup{ if } p_J\notin F,\\
	\displaystyle \inf_{p\in F, p \geq p_J}\|(S \otimes 1_{pp_J^{-1}})\sigma_{pp_J^{-1}}^{-1}\| & \textup{ if } p_J \in F .
\end{cases}
\] 
It follows from Lemma~\ref{lem-cont-of-bund} that $\chi$ is upper semicontinuous. 
 Note that for $F \in U_J$, 
\[
\phi(F)=\Big\|\sum_{j \in J}i_{p_j}^{(F,g)}(T_j)\Big\|=\Big\|\sum_{j \in J}i_{p_J}^{(F,g)}((T_j \otimes 1_{s_j})\sigma_{s_j}^{-1})\Big\|=\|i_{p_J}^{(F,g)}(S)\|=\chi(F).
\]
 Hence, $\phi|_{U_J}$ is upper semicontinuous. If $J$ is empty, then $\phi|_{U_J}=0$. Since $\{U_J: J \subset \{1,2,\cdots,n\}\}$ is an open cover of $\Omega$, it follows that $\phi$ is upper semicontinuous. The proof of $(2)$ is now over. 

 By the definition of $\E_{(F,g)}$, the set $\{f_{p,g,T}(F,g): p \in P, T \in \K_B(E,X_p)\}$ which coincides with $\big\{i_p^{(F,g)}(T): p \in d_{(F,g)}, T \in \K_B(E,X_p)\big\}$ is total in $\E_{(F,g)}$. Hence, $(3)$ follows.  The existence of the topology follows by applying \cite[Thm. C.25]{Dana-Williams}. (The footnote on Page 10 is applicable here as well).
 \end{proof}
 
\begin{lemma}
The left action of \(\A\) on \(\E\) is continuous.
\end{lemma}
\begin{proof}
For $(F,g) \in \G$ and for $p \in d_{(F,g)}$, we denote the natural map $\K_B(X_{pg},X_p) \to \varinjlim_{q \in d_{(F,g)}}\K_{B}(X_{qg},X_q)$ by $j_{p}^{(F,g)}$.  
For \(p,q\in P\), \(T\in \K_B(X_q,X_p)\), define a  section $\phi_{p,q,T}$ of the bundle $\A$ by
\[
\phi_{p,q, T}(F,h) = \begin{cases}
		j_{p}^{(F,g)}(T) & \textup{ if }  p\in d_{(F,g)} ~~\textup{and $g=p^{-1}q$},\\
	0 & \textup{ otherwise. } 
	\end{cases}
\]
Let 
\begin{align*}
    \Gamma_0&:=\textup{span}\{\phi_{p,q,T}:p,q \in P, T \in \K_B(X_q,X_p)\}, \textrm{~and~}\\
    \Gamma_1&:=\textup{span} \{f_{p,g,S}: p \in P, g \in G, S \in \K_B(E,X_p)\}.
\end{align*}
Then, by construction,  $\Gamma_0$ is dense in $\Gamma_\A$ and $\Gamma_1$ is dense in $\Gamma_\E$. We now use Lemma~\ref{lem-equi-cond-continuity} to prove that the left action of \(\A\) on \(\E\) is continuous. Thus, it suffices to show that $\Gamma_0*\Gamma_1 \subset \Gamma_{\E}$. 

Fix $p,q,r \in P$, $g \in G$, $S \in \K_B(X_q,X_p)$ and $T \in \K_B(E,X_r)$.  Note that $\textup{supp}(\phi_{p,q,S}) \subset \Omega \times \{p^{-1}q\}$ and $\textup{supp}(f_{r,g,T}) \subset \Omega \times \{g\}$. Hence, the convolution is supported in $\Omega \times \{p^{-1}q g\}$. Set $g_0:=p^{-1}qg$. 

Now, for $(F,g_0) \in \G$,
\[
\phi_{p,q,S}*f_{r,g,T}(F,g_0)=1_{F}(q^{-1}p)\phi_{p,q,S}(F,p^{-1}q)f_{r,g,T}(Fp^{-1}q,q^{-1}pg_0).
\]
The RHS vanishes unless $p \in F$, $r \geq g^{-1}$ and $r \in Fp^{-1}q$. Let $s_0,s_1 \in P$ be such that $r\vee q=s_0q$ and $r\vee q=s_1r$. Set $R:=(S \otimes 1_{s_0})\cdot((T \otimes 1_{s_1})\sigma_{s_1}^{-1}) \in \K_B(E,X_{s_0p})$. It can now be checked from the definition of the left action (see Eq.~\eqref{left action definition formula}) that 
\[
\phi_{p,q,S}*f_{r,g,T}=f_{s_0p, g_0, R}. 
\]
Hence, $\Gamma_0*\Gamma_1 \subset \Gamma_\E$. The conclusion follows from Lemma \ref{lem-equi-cond-continuity}. 
\end{proof}
The continuity of the right action is analogous. Therefore, we have proved the following result.
\begin{theorem}
\label{Fell_bundle_equivalence_theorem}
Under the assumptions (\textbf{C1})-(\textbf{C3}), the upper semicontinuous Banach bundle \(\E\) implements an equivalence of Fell bundles between \(\A\) and \(\B\).
\end{theorem}

\begin{remark}
Since \(\A\) and \(\B\) are equivalent as Fell bundles,~\cite[Thm. 14]{Sims-Williams-Equivelence-thm-red-version} says that the Fell bundle \(C^*\)-algebras \(C^*_{\red}(\G;\A)\) and \(C^*_{\red}(\G;\B)\) are Morita equivalent. 
It was proved in  \cite[Thm. 5.1]{Rennie_Sims} that, at least when $\G$ is amenable, $C^{*}_{\red}(\G;\A)$ is isomorphic to the Nica--Toeplitz algebra which coincides with the reduced $C^{*}$-algebra of $X$, and it was proved in \cite[Thm.~4.3]{Sundar_Khoshkam} that $C^{*}_{\red}(\G;\B)$ is isomorphic to $\K_{B}(E)\rtimes_{\red} P$. One of the main results of \cite[Thm. 1.1]{Amir_Sundar-product-system} states that $\K_{B}(E)\rtimes_{\red} P$ and the reduced $C^{*}$-algebra of $X$ are Morita equivalent. The above theorem reconciles these pictures and says that the Morita equivalence is indeed implemented by a Fell bundle equivalence.  
\end{remark}

\paragraph{\itshape Acknowledgements:}
We are grateful to the anonymous referee for carefully reading the manuscript and for providing valuable comments and suggestions, which have subsequently improved the clarity and readability of the paper.


\begin{thebibliography}{10}
	
	\bibitem{Anantharaman}
	Claire~Anantharaman-Delaroche and Jean~Renault.
	\newblock {\em Amenable groupoids}, volume~36 of {\em Monographies de
		L'Enseignement Math\'ematique [Monographs of L'Enseignement Math\'ematique]}.
	\newblock L'Enseignement Math\'ematique, Geneva, 2000.
	\newblock With a foreword by Georges Skandalis and Appendix B by E. Germain.
	
	\bibitem{Anantharaman-Delaroche-exact-gpd}
	Claire~Anantharaman-Delaroche.
	\newblock Exact groupoids.
	\newblock https://hal.science/hal-01316189/document.
		\bibitem{Amir_Sundar-product-system}
	Md~Amir Hossain and Sundar Shanmugasundaram.
	\newblock {R}educed {$C^*$}-algebras of product systems---an {$E_0$}-semigroup
	and a groupoid perspective. \newblock {\em J. Noncommut. Geom.} DOI 10.4171/JNCG/675, 2026. 
	
	\bibitem{IW-2012-Ideal-st}
	Marius Ionescu and Dana~P. Williams.
	\newblock Remarks on the ideal structure of {F}ell bundle {$C^\ast$}-algebras.
	\newblock {\em Houston J. Math.}, 38(4):1241--1260, 2012.
	
	\bibitem{Lalonde}
	Scott~M. LaLonde.
	\newblock Nuclearity and exactness of groupoid crossed products.
	\newblock {\em Journal of Operator theory}, 74:213--245, 2015.
	
	\bibitem{Lalonde-2016-Equivalence-exact}
	Scott~M. LaLonde.
	\newblock Equivalence and exact groupoids.
	\newblock {\em Houston J. Math.}, 42(4):1267--1290, 2016.
	
	\bibitem{Lalonde-stabilization-thm}
	Scott~M. LaLonde.
	\newblock Some consequences of stabilization theorem for {F}ell bundles over
	exact groupoids.
	\newblock {\em J. Operator Theory}, 81(2):335--369, 2019.
	
	\bibitem{Lalonde-Some-permanenece-property-of-exact}
	Scott~M. LaLonde.
	\newblock On some permanence properties of exact groupoids.
	\newblock {\em Houston J. Math.}, 46(1):151--187, 2020.
	
	\bibitem{Xin-Li-2012-Nucl-semigroup-alg-amen}
	Xin Li.
	\newblock Nuclearity of semigroup {$C^*$}-algebras and the connection to
	amenability.
	\newblock {\em Adv. Math.}, 244:626--662, 2013.
	
	\bibitem{Muhly-Williams-Disintegration-theorem}
	Paul~S. Muhly and Dana~P. Williams.
	\newblock Equivalence and disintegration theorems for {F}ell bundles and their
	{$C^*$}-algebras.
	\newblock {\em Dissertationes Math.}, 456:1--57, 2008.
	
	\bibitem{Rennie_Sims}
	Adam Rennie, David Robertson, and Aidan Sims.
	\newblock Groupoid {F}ell bundles for product systems over quasi-lattice
	ordered groups.
	\newblock {\em Math. Proc. Cambridge Philos. Soc.}, 163(3):561--580, 2017.
	
	\bibitem{Sims-Williams-Equivelence-thm-red-version}
	Aidan Sims and Dana~P. Williams.
	\newblock An equivalence theorem for reduced {F}ell bundle {$C^*$}-algebras.
	\newblock {\em New York J. Math.}, 19:159--178, 2013.
	
	\bibitem{Sundar_Khoshkam}
	S.~Sundar.
	\newblock On a construction due to {K}hoshkam and {S}kandalis.
	\newblock {\em Doc. Math.}, 23:1995--2025, 2018.
	
	\bibitem{Dana-Williams}
	Dana~P. Williams.
	\newblock {\em Crossed products of {$C{^\ast}$}-algebras}, volume 134 of {\em
		Mathematical Surveys and Monographs}.
	\newblock American Mathematical Society, Providence, RI, 2007.
	
\end{thebibliography}

\end{document}